\documentclass[reqno,10pt]{amsart}
\usepackage[paperwidth=19.5cm]{geometry}
\usepackage{tikz, enumerate, color}

\numberwithin{equation}{section}
\newtheorem{theorem}{Theorem}[section]
\newtheorem{proposition}[theorem]{Proposition}

\newtheorem{lemma}[theorem]{Lemma}
\theoremstyle{definition}
\newtheorem{remark}[theorem]{Remark}

\newenvironment{enumeratei}{\begin{enumerate}[\quad\upshape (i)]} {\end{enumerate}}

\newcommand\red[1] {\color{red}#1\color{black}}
\newcommand\magenta[1] {\color{magenta}#1\color{black}}

\newcommand\ssc [1] {^{(#1)}} 

\newcommand \tbf [1] {{\bf #1}}
\newcommand \length {\ell}
\newcommand \set [1] {\{#1\}}
\newcommand \br {\textup{br}}
\newcommand \bxi {\xi_{\textup B}} 
\newcommand \pxi {\xi_{\textup P}} 

\newcommand \kb {\kern-6pt}
\newcommand \eqeqref[1]{\overset{\eqref{#1}}= }
\newcommand \lesseqref[1]{\overset{\eqref{#1}}< }
\newcommand \nonparallel {\mathrel{\not{\kern-1.3pt\parallel}}}
\newcommand \datum {November 5th, 2019}

\begin{document}

\title
[{Medians are below joins}]
{Medians are below joins in semimodular lattices of breadth 2}

\author[G. Cz\'edli]{G\'abor Cz\'edli}
\email{czedli@math.u-szeged.hu}
\urladdr{http://www.math.u-szeged.hu/\textasciitilde{}czedli/}
\address{University of Szeged, Bolyai Institute, Szeged, HUNGARY 6720}

\author[R. C. Powers]{Robert C. Powers}
\email{rcpowe01@louisville.edu}
\address{Department of Mathematics,
University of Louisville, Louisville, Kentucky 40292 USA}

\author[J. M. White]{Jeremy M. White}
\email{jwhite07@spalding.edu}
\address{School of Natural Science,
Spalding University, Louisville, Kentucky 40203 USA}

\date{\red{\hfill \tbf{\datum}}}

\begin{abstract}
Let $L$ be a lattice of finite length and let $d$ denote the minimum path length metric on the covering graph of $L$. For any $\xi=(x_1,\dots,x_k)\in L^k$, an element $y$ belonging to $L$ is called a {\em median} of $\xi$ if the sum $d(y,x_1)+\cdots+d(y,x_k)$ is minimum. The lattice $L$ satisfies the \emph{$c_1$-median property} if, for any $\xi=(x_1,\dots,x_k)\in L^k$ and for any median $y$ of $\xi$, $y\leq x_1\vee\dots\vee x_k$. Our main theorem asserts that if $L$ is an upper semimodular lattice of finite length and the breadth of $L$ is less than or equal to $2$,  then $L$ satisfies the $c_1$-median property. Also, we give a construction that yields semimodular lattices, and we use a particular case of this construction to prove that our theorem is sharp in the sense that  $2$  cannot be replaced by $3$. 
\end{abstract}



\thanks{This research was supported by the Hungarian Research, Development and Innovation Office under grant number KH 126581.}

\maketitle

\section{Introduction}
Given a lattice $L$ of finite length and  $\xi=(x_1,\dots,x_k)\in L^k$, an element $y\in L$ is called a {\em median} of $\xi$ if the sum $d(y,x_1)+\cdots+d(y,x_k)$ is minimum, where $d(y,x_i)$ stands for the path distance in the Hasse diagram of $L$. 
 Our goal is to prove that 
\begin{equation}\left.
\parbox{9cm}{
whenever $L$ is, in addition, upper semimodular and of breadth at most 2, to be defined later, then $y\leq x_1\vee \dots\vee x_k$ holds for every $k\geq 2$ and for any median $y$ of every $\xi=(x_1,\dots,x_k)\in L^k$;}\,\,\,\right\}
\label{pbxentree}
\end{equation}
see  our main result, Theorem~\ref{thmmain}, for more details.

\subsection{Outline}\label{subsectoutline} 
The paper is structured as follows. In Subsection~\ref{subsectsurvey}, we survey some earlier results on medians in lattices. Subsection~\ref{subsectbasics} recalls some definitions, whereby the paper is readable with minimal knowledge of Lattice Theory. 
In Section~\ref{sectionexample}, we give a new way of constructing semimodular lattices; see Proposition~\ref{propconstrsm}, which can be of separate interest. As a particular case of our construction, we present a semimodular lattice $L(n,k)$ with breadth $k$ and size $|L(n,k)|=2n^k -(n-1)^k$ for any integers $k\geq 3$ and $n\geq 4$ such that $L(n,k)$ fails to satisfy the $c_1$-median property. Section~\ref{sectwttLs} is devoted to two technical lemmas that will be used later. 
Finally, Section~\ref{sectmainresult} presents our main result, Theorem~\ref{thmmain}, which asserts somewhat more than \eqref{pbxentree}. Using the auxiliary statements proved in Sections~\ref{sectionexample} and \ref{sectwttLs}, Section~\ref{sectmainresult} concludes with the proof of Theorem~\ref{thmmain}.

\subsection{Survey}\label{subsectsurvey}
For any metric space $(X, d)$ and for any $k$-tuple $\xi=(x_1,\dots,x_k)$ belonging to $X^k$, $y \in X$ is called a {\em median of $\xi$} if
\begin{equation}
r(y, \xi) = \sum_{i=1}^k d(y,x_i)
\label{eqRmTnSs}
\end{equation}
is minimum. Medians are frequently used numerical attributes of, say,  (discrete) probability distributions, and they are interesting in other areas of mathematics and even outside mathematics; see, for example, Monjardet~\cite{M}.

The $k$-tuple $\xi$ above is called a {\em profile} and $\{\xi\}$ denotes the set of all elements belonging to the profile. Repetition among the $x_i$'s is permitted so $|\{\xi\}| \leq k$. The notation $M(\xi)$ is used for the set of all medians of $\xi$ and $r(y, \xi)$ is called the \emph{remoteness} of $y$ from $\xi$. One can view $M$ as a function with domain the set of all possible profiles and range the set of all nonempty subsets of $X$. In this case, $M$ is called the {\em median function} or the {\em median procedure}. The median function has been extensively studied and we refer the reader to Day and McMorris~\cite{dm} for more information about this function.

If $X$ is a lattice $L$ of finite length and $d$ is the minimum path length metric on the covering graph of $L$, then it is sometimes possible to describe a median set $M(\xi)$ explicitly. For example, if $L$ is a finite distributive lattice and $\xi=(x_1,\dots,x_k)\in L^k$, then
\begin{align*} M(\xi) &= [m(\xi), m^{\prime}(\xi)] = \{z \in L : m(\xi) \leq z \leq m^{\prime}(\xi)\}\text{ where}\cr
m(\xi) &= \bigvee \{\ \bigwedge_{i \in I} x_i : I \subseteq \{1, \ldots, k\}, |I| \geq \frac{k}{2} + 1\}\text{ and}\cr
m^{\prime}(\xi) &= \bigwedge \{\ \bigvee_{i \in I} x_i : I \subseteq \{1, \ldots, k\}, |I| \geq \frac{k}{2} + 1\}.
\end{align*}
This result is due to Barbut~\cite{bar} and Monjardet~\cite{M}. 
Their result was extended by Bandelt and Barth\-\'{e}lemy to median semilattices \cite{bb}. In addition, Barth\'{e}lemy showed that $M(\xi)$ is a sublattice of the \emph{interval} $[m(\xi), m^{\prime}(\xi)]$ if $L$ is a finite modular lattice \cite{jp}. In the case where $L$ is assumed to be a finite upper semimodular lattice, Leclerc \cite{lec} proved that $M(\xi)\subseteq [m(\xi), 1_L]$ for every $\xi\in L^k$. 
Leclerc also showed the converse. Specifically, if a finite lattice $L$ has the property that $M(\xi)\subseteq [m(\xi), 1_L]$ for every $\xi\in L^k$, then $L$ is upper semimodular. Leclerc's work was generalized to finite upper semimodular posets in \cite{pow}.

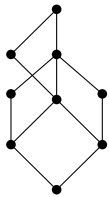
\begin{figure}[hbt]
\begin{tikzpicture}[scale=0.3]
\draw (2,0) coordinate (A);
\draw (0,2) coordinate (B);
\draw (4,2) coordinate (C);
\draw (2,4) coordinate (D);
\draw (0,4.25) coordinate (E);
\draw (4,4.25) coordinate (F);
\draw (0,6) coordinate (I);
\draw (2,6) coordinate (H);
\draw (2,8) coordinate (G);
\draw (A)--(B);
\draw (A)--(C);
\draw (B)--(D);
\draw (B)--(E);
\draw (C)--(D);
\draw (C)--(F);
\draw (E)--(H);
\draw (F)--(H);
\draw (D)--(H);
\draw (D)--(I);
\draw (G)--(H);
\draw (G)--(I);
\fill[black] (A) circle (6pt);
\fill[black] (B) circle (6pt);
\fill[black] (C) circle (6pt);
\fill[black] (D) circle (6pt);
\fill[black] (E) circle (6pt);
\fill[black] (F) circle (6pt);
\fill[black] (G) circle (6pt);
\fill[black] (H) circle (6pt);
\fill[black] (I) circle (6pt);
\end{tikzpicture}
\caption{A nonplanar  semimodular lattice of breadth two}
\label{fignpl}
\end{figure}

While  Leclerc \cite{lec} above gives a lower bound of $M(\xi)$, here we are interested in a reasonable upper bound. Namely, 
following White \cite{whi}, we will say that a lattice $L$ satisfies the {\em $c_1$-median property} if, for any $\xi=(x_1,\dots,x_k)\in L^k$,
\begin{equation}
y \leq c_1(\xi) := \bigvee _{i=1}^k x_i
\label{eqc1xi}
\end{equation}
for all $y\in M(\xi)$. This is obviously equivalent to $\bigvee M(\xi) \leq c_1(\xi)$. 
Since $m^{\prime}(\xi) \leq c_1(\xi)$ for all $\xi$, it follows that every finite \emph{modular} lattice satisfies the $c_1$-median property. 
Finite (upper) semimodular lattices are known to be graded.  (As usual,  ``semimodular'' will always mean ``upper semimodular''.)  Cz\'edli, Powers, and White~\cite{cpw} proved that
\begin{equation}
\text{every planar graded lattice satisfies the $c_1$-median property.}\label{eqtxtczpwh}
\end{equation}
Let us emphasize that a planar lattice is finite by definition; see  Gr\"atzer and Knapp~\cite[page 447]{GrKnI} or
  Cz\'edli and Gr\"atzer~\cite[page 92]{czgg}.
Clearly, \eqref{eqtxtczpwh} implies immediately that 
\begin{equation}\left.
\parbox{6cm}{planar semimodular lattices satisfy the $c_1$-median property.}\,\,\,\right\}
\label{eqpbxfrMczpWh}
\end{equation}
It belongs to the folklore and we will prove in  Section~\ref{sectwttLs} that
\begin{equation}
\text{every planar lattice is of breadth at most $2$.}
\label{eqtxtbr2planar}
\end{equation}
Hence  \eqref{pbxentree} is a generalization of \eqref{eqpbxfrMczpWh}. Furthermore, this is a proper generalization since there are non-planar finite semimodular lattices of breadth 2; see Figure~\ref{fignpl} for an example.
Note at this point that the class of 
all semimodular lattices of finite length and breadth 2 is plentiful since, for example,  Rival \cite{riv} proved that this class contains lattices with arbitrarily large finite width and length. Note also that a graded lattice need not be semimodular,  and so it is easy to see that none of  \eqref{pbxentree} and \eqref{eqtxtczpwh} implies the other one.

In 2000, Li and Boukaabar \cite{li} gave a semimodular lattice  with $101$ elements that fails to satisfy the $c_1$-median property; we will denote this lattice by $L_{\textup{LiBou}}$. Hence,  \eqref{pbxentree} cannot be extended to all semimodular lattices of finite length.  Our Theorem~\ref{thmmain} will assert even more: as $L(n,3)$ in Section~\ref{sectionexample} exemplifies, \eqref{pbxentree} cannot be extended  to finite length semimodular lattices of breadth 3. Note that Section~\ref{sectionexample} builds on the essence of $L_{\textup{LiBou}}$ but, in addition that we will show that $L(n,3)$ is of breadth 3, 
there is a significant difference between the two approaches. Namely, as opposed to \cite{li}, where
$L_{\textup{LiBou}}$ is defined by its  involved Hasse diagram, tedious work is needed to show that it is a lattice and it is semimodular, and most of this work is left to the reader, our argument proving the same properties of $L(n,3)$ does not rely on any diagram and it is easy to read.

It was proved in  White~\cite{whi} that 
\begin{equation}\left.
\parbox{7cm}{semimodular lattices of height at most 6 satisfy the $c_1$-median property.}
\,\,\,\right\}
\label{pbxWhhght6} 
\end{equation}
Each of the conditions given in  \eqref{pbxentree}, \eqref{eqtxtczpwh}, \eqref{eqpbxfrMczpWh}, and \eqref{pbxWhhght6}
determines an interesting class of semimodular lattices of finite length satisfying the $c_1$-median property. 
Although interesting additional such classes of semimodular lattices will hopefully be discovered in the future, we do not see much hope for a reasonable characterization of semimodular lattices of finite length that satisfy the $c_1$-median property.

\subsection{Basic concepts}\label{subsectbasics}
All the elementary concepts and notation not defined in this paper can easily be found in Gr\"atzer~\cite{ggbypicture}
or in its freely downloadable \emph{Part I. A Brief Introduction to Lattices and Glossary of Notation}  at \texttt{tinyurl.com/lattices101}, and also in Nation~\cite{nationbook}, freely available again. Alternatively, the reader can look into Davey and Priestley~\cite{dp} or  Stern~\cite{ste}.  
However, for convenience, we recall the following. 
A lattice $L$ is of \emph{finite length} if there is a 
nonnegative integer $n$ such that every chain of $L$ consists of at most $n+1$ elements; if so, then the smallest such $n$ is the \emph{length} of the lattice, denoted by $\length(L)$. 
A lattice of finite length is \emph{graded} if any two of its maximal chains have the same (finite) number of elements. A lattice $L$  is \emph{upper semimodular}, or simply \emph{semimodular}, if for every $x,y\in L$, the covering $x\wedge y\prec x$ implies $y\prec x\vee y$. The condition \emph{lower semimodular} is defined dually. It is well known that every semimodular lattice of finite length is graded. For $x,y\in L$, the distance between $x$ and $y$ in the undirected covering graph associated with $L$ is denoted by $d(x,y)$. 
It is straightforward to see that in a \emph{semimodular} lattice $L$ of finite length, for any $x,y,u,v,w\in L$, 
\begin{align}
&d(x,y)=d(x,x\vee y)+d(x\vee y,y)=\length([x,x\vee y])+\length([y,x\vee y])
\label{eqPndskszD}\\
&\text{and  }u\leq v\leq w\text{ implies that } d(u,w)=d(u,v)+d(v,w).
\label{alignchainddd}
\end{align}
The \emph{breadth} of a lattice $L$, to be denoted by $\br(L)$, is the least positive integer $n$ such that any join $\bigvee_{i=1}^m x_i$, $x_i\in L$, $m\geq n$, is always a join of $n$ of the joinands $x_i$.

\section{Semimodular constructs and an example}\label{sectionexample}

An element $u$ in a lattice $L$ is  \emph{join-irreducible} if for every $x,y\in L$, $u=x\vee y$ implies that $u=x$ or $u=y$. Similarly, if  $u\leq x\vee y$ implies that $u\leq x$ or $u\leq y$, then $u$ is  \emph{join-prime}. Finally, $u$ is \emph{codistributive} (or \emph{dually distributive})  if for every $x,y\in L$,   
$\,u\wedge (x\vee y)=(u\wedge x)\vee (u\wedge y)$; see, for example, \v{S}e\v{s}elja and Tepav\v{c}evi\v{c}~\cite{seseljatepavcevic} and Gr\"atzer\cite{ggfoundbook}. 
Clearly, a join-prime element is join-distributive. If an element is codistributive and join-irreducible, then it is  join-prime; see (the easy proof of) Nation~\cite[Theorem 8.6(1)]{nationbook}. So there are many examples of join-prime elements in lattices. Note that each of the three free generators of the 28-element free modular lattice is join-prime, join-irreducible, but not codistributive; see Gr\"atzer~\cite[Figure 20 in page 85]{ggfoundbook}.
Observe that, for every positive integer $t$ and any lattices $K_1$, \dots, $K_t$ of finite length, 
\begin{equation}\left.
\parbox{10cm}{a nonzero element $e=(e_1,\dots,e_t)\in K_1\times\dots\times K_t$ is join-prime if and only if there exists a unique $i=i(e)\in\set{1,\dots,t}$ such that $e_i$ is a nonzero join-prime element of $K_i$ and $e_j$ is the bottom element $0_j$ of $K_j$ for all $j\in\set{1,\dots,t}\setminus\set i$.}\,\,\right\}
\label{pbxzTdsmKpclFhZ}
\end{equation}
In order to verify \eqref{pbxzTdsmKpclFhZ}, assume that $e$ has at least two nonzero coordinates, say, $e_1$ and $e_2$.  Then $e\leq (e_1,0_2,\dots.0_t)\vee (0_1,e_2,\dots, e_t)$ witnesses that $e$ is not join-prime. The rest of the argument proving \eqref{pbxzTdsmKpclFhZ} is even more trivial and will not be detailed.

\begin{proposition}\label{propconstrsm} 
Let $K$ be a lattice of finite length.
\begin{enumeratei}
\item\label{lemmaconstrsma}
If $e$ is a nonzero join-prime element of  $K$, $f\in K$, and $e\leq f$, then the subposet $L:=K\setminus[e,f]$ of $L$ is a lattice.
\item\label{lemmaconstrsmb} If $t$ is a positive integer, 
 $K_1$, \dots, $K_t$  are  semimodular lattices of finite length, $K=K_1\times\dots\times K_t$ is their direct product, $e=(e_1,\dots, e_t)\in K$ is a nonzero join-prime element, $i=i(e)$ denotes the subscript defined  in \eqref{pbxzTdsmKpclFhZ}, and  $f=(f_1,\dots, f_t)$ is an element of $K$ such that $f_i$ is the top element $1_i$ of $K_i$,
then the subposet $L:=K\setminus[e,f]$ of $K$ is a semimodular lattice, and it is a join-subsemilattice of $K$.
\end{enumeratei}
\end{proposition}

Note that \eqref{pbxzTdsmKpclFhZ} and the assumptions of part \eqref{lemmaconstrsmb} above imply that $e\leq f$, whereby the interval $[e,f]$ in  \eqref{lemmaconstrsmb} makes sense. Note also that the case $t=1$ is also interesting, but this case would be  easier to prove than the general case  $t\in \set{1,2,3,\dots}$. 

\begin{proof} First, we are going to prove \eqref{lemmaconstrsma}. Since $0_K<e$,  
the subposet $L$ has a least element,  
$0:=0_K$. Observe that $L$ is of finite length since so is $K$.  Thus, to prove that $L$ is a lattice, it suffices to prove that $L$ is join-closed. So it suffices to show that $L$ is a join-subsemilattice of $K$. Suppose, for a contradiction,  that $x,y\in L$ but $x\vee y\notin L$.
Then $e\leq x\vee y\leq f$. Since $e$ is join-prime, we obtain that $e\leq x$ or $e\leq y$, and we can assume that $e\leq x$ by symmetry. This with  $x\leq x\vee y\leq f$ lead to $x\in [e,f]$, contradicting $x\in L$. 
Thus, $L$ is join-closed and part \eqref{lemmaconstrsma} holds.

Next, we turn our attention to \eqref{lemmaconstrsmb}. 
We can assume that $i=1$. Then, by \eqref{pbxzTdsmKpclFhZ},
\begin{equation}
e_1>0_1,\,\, e_2=0_2,\,\,\dots, \,\, e_t=0_t.
\label{eqZhWrsMBFs}
\end{equation}
We obtain from  part \eqref{lemmaconstrsma} that $L$ is a lattice.   We are going to show that
\begin{equation}\left.
\parbox{6.3cm}{whenever $\set{x,y}\subseteq L$ and $y$ covers $x$ in $L$,
then $y$ covers $x$ in $K$.}\,\,\,\right\}
\label{pbxcvprsXt}
\end{equation}
First of all, observe that 
for any $a,b\in K$, we trivially have that 
\begin{equation}\left.
\parbox{9.2cm}{$a\prec_K b$ if and only if $a_j\prec b_j$ for exactly one subscript $j$ and $a_s=b_s$ for every other subscript $s$;  note that this holds even if $K_1$, \dots, $K_t$ are not assumed to be semimodular.}
\,\,\,\right\}
\label{pbxPrCrlPr}
\end{equation}
For the sake of contradiction, suppose that $x\prec_L y$ but $x\not\prec_K y$. Then there is at least one element in $[e,f]\cap [x,y]$. Hence, for $a:=e\vee x$ and $b:=f\wedge y$, we have that $a\leq b$. 
Note that $x\leq a\leq b\leq f$, so $x\notin [e,f]$ yields that $e\not\leq x$. Similarly, $e\leq a\leq b\leq y$ and $y\notin [e,f]$ give that $y\not\leq f$. 
Since $a\in [e,f]$ but $x\notin[e,f]$, we have that $x<a$. If we had an $x'\in K$ such that $x<x'<a$, then $x<x'<a\leq b<y$ and $x\prec_L y$ would imply that $x'\notin L$, whereby $e\leq x'$ would lead to the contradiction
$a=e\vee x\leq x'<a$. Thus, $x\prec_K a$ in $K$. Similarly, $b\prec_K y$. Let us summarize:
\begin{equation}\left.
\begin{aligned}
x&\prec_K x\vee e=a\leq b=y\wedge f\prec_K y,\,\, 
\cr
e&\not\leq x,\,\,\, y\not\leq f,\,\,\,
e\leq y,\,\,\, x\leq f .
\end{aligned} \,\,\,\right\}
\label{alignedZrWt}
\end{equation}
Since $e\not\leq x$, \eqref{eqZhWrsMBFs} gives that $e_1\not\leq x_1$. We know from \eqref{alignedZrWt} that 
$x\leq f$, and so  we obtain that $x_2\leq f_2$, \dots, $x_t\leq f_t$. Hence, if we had that $x_2=y_2$, \dots, $x_t=y_t$, then we would get that $y\leq f$ since $f_1=1_1$, but $y\leq f$ would contradict \eqref{alignedZrWt}.  Thus, there is a subscript $j\in\set{2,\dots,t}$ such that $x_j<y_j$. By symmetry, we can assume that $j=2$, that is, $x_2<y_2$.  Take the element $z:=(x_1,y_2,x_3,\dots,x_t)$ in $K$. Since $e_1\not\leq x_1=z_1$,  we have that $e\not\leq z$, whereby $z\in L$. Using $x_2<y_2=z_2$, we obtain that $x<z$. Since $x<y$, we have that $z\leq y$. 
Using that $e_1\not\leq x_1=z_1$ but  \eqref{alignedZrWt} 
gives that $e_1\leq y_1$, it follows that  $z\neq y$. So $z<y$. 
Since $x<z$, $z<y$, and $z\in L$ contradict $x\prec_L y$, we conclude \eqref{pbxcvprsXt}.
 
Next, recall from Cz\'edli and Walendziak~\cite{czw} that
\begin{equation}
\text{the direct product of finitely many semimodular lattices is  semimodular.}
\label{eqtxtWsrptRdzd}
\end{equation}
This yields that $K$ is semimodular. This fact, \eqref{pbxcvprsXt}, and Exercise 3.1 in \cite{czgg} imply the semimodularity of $L$. 
This proves part  \eqref{lemmaconstrsmb} and completes the proof of Proposition~\ref{propconstrsm}.
\end{proof}

\begin{lemma}\label{lemmaProDct} 
For any integer $t\geq 2$ and non-singleton lattices $L_1,\dots, L_t$ of finite breadth,
\[\br(L_1\times\dots\times L_t)=\br(L_1)+\dots+\br(L_t).\]
\end{lemma}

Having no reference at hand, we present a straightforward proof of this easy lemma.

\begin{proof} We can assume that $t=2$, because then the lemma follows by induction. 
For $i\in\set{1,2}$, denote $\br(L_i)$ by $n_i$, and pick an $n_i$-element subset $\set{a\ssc i_1, \dots, a\ssc i_{n_i}}$ of $L_i$ such that no element of this subset is the smallest element of $L_i$ (which need not exist), and $b\ssc i:= a\ssc i_1\vee \dots\vee a\ssc i_{n_i}\in L_i$ is an \emph{irredundant join}, that is, none of the joinands can be omitted without making the equality false.
Pick $c\ssc i\in L_i$ such that $c\ssc i<b\ssc i$ and  $c\ssc i\leq a\ssc i_j$ for all $j\in\set{1,\dots, n_i}$; this is possible either because $n_i>1$ and we can let $c\ssc i=a\ssc i_1\wedge \dots\wedge a\ssc i_{n_i}$, or because $n_i=1$ and we can pick an element smaller than $a\ssc i_1$. Since the join $(b\ssc1,b\ssc2)$  of the elements 
$(a\ssc1_1, c\ssc2)$, $(a\ssc1_2, c\ssc2)$, \dots, $(a\ssc1_{n_1}, c\ssc2)$, 
$(c\ssc1, a\ssc2_1)$, $(c\ssc1, a\ssc2_2)$, \dots, $(c\ssc1, a\ssc2_{n_2})$ 
is clearly an irredundant join, $\br(L_1\times L_2)\geq n_1+n_2=\br(L_1)+\br(L_2)$. 
To prove the converse inequality, assume that $(w_1,w_2)=\bigvee S$ in $L_1\times L_2$ with $|S|\geq n_1+n_2$. For each $i\in\set{1,2}$, we can pick an $n_i$-element subset $T_i$ of $S$ such that $w_i=\bigvee_{v\in T_i}v_i$. Letting $T$ be an $(n_1+n_2)$-element subset of $S$ such that $T_1\cup T_2\subseteq T$, we have that $(w_1,w_2)\leq \bigvee T\leq \bigvee S=(w_1,w_2)$. Thus,  $\br(L_1\times L_2)\leq n_1+n_2=\br(L_1)+\br(L_2)$.  
\end{proof}

For integers $n\geq 4$ and $k\geq 3$, we define a lattice {$L(n,k)$ as follows. 
Let $C_n=\{0,1,2,\ldots, n-1\}$  be the $n$-element chain with the usual ordering from $\mathbb Z$. Let $K=K(n,k)$ be the $(k+1)$-fold direct product 
\[K=K(n,k)=C_n\times C_n\times \dots \times C_n\times C_2.
\]
After defining $e=(e_1,\dots,e_{k+1})$ and $f=(f_1,\dots,f_{k+1})$ 
by
\[e:=(0,\ldots, 0,1,0) \text{ and } f:=(n-2,\ldots, n-2, n-1,0),
\]
we define $L=L(n,k)$ as $K\setminus [e,f]$. At present, $L(n,k)$ is only a poset.

\begin{lemma}\label{lemmaProdBr}
For integers $n\geq 4$ and $k\geq 3$, $L(n,k)$ is a $(2n^k -(n-1)^k)$-element semimodular lattice of breadth $k$, and this lattice fails to satisfy the $c_1$-median property.
\end{lemma}

\begin{proof}
In a chain, every  element is join-prime. Thus, it follows from Proposition~\ref{propconstrsm} that 
$L=L(n,k)$ is a semimodular lattice. Clearly,
$|L|=|K|-|[e,f]| = 2n^k -(n-1)^k$. 

The $2^k$-element boolean lattice is isomorphic to, say, $\set{2,3}\times\dots\times \set{2,3}\times\set 1$, which is a join-subsemilattice of $L$. Hence, 
we obtain from Lemma~\ref{lemmaProDct} (or we conclude easily even without this lemma) that $\br(L)\geq k$. 
\magenta In order to prove the converse inequality, 
let $\mathcal W=\{w\ssc1,w\ssc2,\ldots,w\ssc m\}$ with $m\geq k+1$ be a collection of elements from $L$. Denote $\bigvee \mathcal W$ by $y$. It suffices to find an at most $k$-element subset $\mathcal W^\ast$ of $\mathcal W$ such that $\bigvee \mathcal W^\ast=y$. 
For each $i=1,\ldots,k+1$, we can find at least one $w\ssc {j_i}\in\mathcal W$ such that $y_i=w\ssc {j_i}_i$. Let $\mathcal W':=\set{w\ssc {j_1}, \dots,w\ssc {j_{k+1}}}$. Clearly, $\bigvee \mathcal W'=y$ and $|\mathcal W'|\leq k+1$. Suppose that $y_i=0$ for some $i\in\set{1,\dots, k+1}$. Then $\bigvee (\mathcal W'\setminus\set{w\ssc{j_i}})$ still equals $y$, so  
$\mathcal W'\setminus\set{w\ssc{j_i}}$ serves as $\mathcal W^\ast$. Now assume that every coordinate of $y$ is nonzero;
in particular, $y_{k+1}=1$.
We can also assume that $w\ssc {j_k}_{k+1}=0$ since otherwise the equality $w\ssc {j_k}_{k+1}=1$ would make $w\ssc{j_{k+1}}$ superfluous, that is, we could let $\mathcal W^\ast:=\mathcal W'\setminus\set{w\ssc{j_{k+1}}}$.
Since $w\ssc {j_k}_{k}=y_k\neq 0$ gives that $e\leq w\ssc {j_k}$ but $w\ssc {j_k} \notin [e,f]$, it follows that $w\ssc {j_k}\not\leq f$. This fact and $w\ssc {j_k}_{k+1}=0$ give that
$w\ssc {j_k}_{i}=n-1$ for some $i\in\set{1,\dots, k-1}$. So $n-1=w\ssc {j_k}_{i}\leq y_i=w\ssc {j_i}_{i}$, where the inequality turns into an equality since $n-1$ the largest element of $C_n$. Thus, we can let $\mathcal W^\ast:= \mathcal W'\setminus\set{w\ssc{{j_i}}}$. We have proved that $\br(L)=k$.

Next, to prove that $L$  does not satisfy the $c_1$-median property, let

\begin{equation}\left.
\begin{matrix}
x\ssc 0 =(& 0, &0,& 0,&\ldots,&0,&0,&0&),\cr
x\ssc 1=(&n-1,& 0,& 0,&\ldots,&0,&n-1,&0&), \cr
x\ssc 2=(&0,& n-1,&0,&\ldots,&0,&n-1,&0&),
\end{matrix}\,\,\right\}
\label{eqmtxSxTdnQ}
\end{equation}
and define $\xi:=(x\ssc 0, x\ssc 1,x\ssc 2)\in L^3$. Clearly, $c_1(\xi)=(n-1,n-1,0,\ldots,0,n-1,0)$; see \eqref{eqc1xi}.  By \eqref{eqRmTnSs} and \eqref{eqPndskszD}, the remoteness of an arbitrary $y=(y_1,y_2,\ldots,y_k,y_{k+1})\in L$ with respect to $\xi$ is
\begin{align}
r(y,\xi) & =   \sum_{i=1}^2 [ (n-1) - y_i + 2y_i] + \sum_{i=3}^{k-1} 3y_i + 2(n-1)-y_{k} \cr
&+3y_{k+1}  = 4(n-1) + y_1 + y_2 -y_{k} + 3y_{k+1} + \sum_{i=3}^{k-1} 3y_i.
\label{alignremotey}
\end{align}
Consider $z=(0,0,0,\ldots,0,n-1,1)\in L$. By \eqref{alignremotey} or trivially,
\begin{equation}
r(z,\xi)=2(n-1)+n-1+3=3n.
\label{eqsXi}
\end{equation}
We are going to show that, for every $y\in K=K(n,k)$, 
\begin{equation}
\text{$r(y,\xi)<r(z,\xi)$ implies $y\not\in L$.}
\label{eqtxtyzimpl} 
\end{equation}
Suppose that $r(y,\xi)<r(z,\xi)$. Thus, using
$y_k\leq n-1$,  \eqref{alignremotey}, and \eqref{eqsXi}, we obtain after rearranging and simplifying that
\begin{equation}
n + y_1 + y_2  + 3y_{k+1} + \sum_{i=3}^{k-1} 3y_i <  y_{k}+4 \leq n-3.
\label{rem2}
\end{equation}
This implies that
$y_1 + y_2 + 3\cdot\bigl(y_{k+1}+ \sum_{i=3}^{k-1}y_i \bigr)  <  3$, whereby 
\begin{equation}\left.
\parbox{6.5cm}{$y_i=0$ for $i\in\set{3,4\dots,k-1,k+1}$ and 
$y_i\leq 2 \leq n-2$ for $i=1,2$.}\,\,\,\right\}
\label{eqtxtyikp1}
\end{equation}
The first inequality in (\ref{rem2}) together with $n\geq4$  yield that  that $1\leq y_{k}$. This fact and \eqref{eqtxtyikp1} imply that $y\in [e,f]$, that is, $y\notin L$. Consequently, \eqref{eqtxtyzimpl} holds, and so $z\in M(\xi)$. Since $z\not\leq c_1(\xi)$, it follows that $L$ does not satisfy the $c_1$-median property.
\end{proof}

For  lattices $(L';\leq')$ with top $1'$ and $(L'';\leq'')$ with bottom $0''$, their \emph{glued sum} 
is defined to be $\bigl((L'\setminus\set{1'})\cup\set{1'=0''} \cup( L''\setminus\set{0''});\leq\bigr)$ where $x'\leq y''$ for any $(x',y'')\in L'\times L''$ and  the restriction of $\leq$ to $L'$ and that to $L''$ are $\leq'$ and $\leq''$, respectively. Saying in a pragmatical way for the finite case: we put the diagram of $L''$ atop that of $L'$ and we identify $1'$ with $0''$. For example, the glued sum of the 2-element chain and the 3-element chain is the 4-element chain.
The following remark is a trivial consequence of the case $(n,k)=(4,3)$ of Lemma~\ref{lemmaProdBr}; note that the proof of this particular case would not be significantly shorter than that of  Lemma~\ref{lemmaProdBr}.

\begin{remark}\label{remarkGlS}
For $k>3$, we can easily construct a finite semimodular lattice $G(k)$ of breadth $k$ such that $G(k)$ does not satisfy the $c_1$-median property and its size is  less than $|L(4,k)|=2\cdot 4^k-3^k$.
Namely, let $G(k)$ be the glued sum of $L(4,3)$ and the $2^k$-element boolean lattice; its size is $|G(k)|=2\cdot 4^3-3^3+2^k-1=2^{k}+100$.
\end{remark}

\section{Two technical lemmas}\label{sectwttLs}
Before formulating two technical lemmas, we prove \eqref{eqtxtbr2planar}, simply because we could not find any reference to this almost trivial statement. 
\begin{proof}[Proof of  \eqref{eqtxtbr2planar}]
For the sake of contradiction, suppose that $L$ is a planar lattice but not of breadth at most 2. Then we can take a join $x_1\vee \dots \vee x_n=:y$ in $L$ 
such that $n\geq 3$ but $y\neq x_i\vee x_j$ for any $i,j\in\set{1,\dots, n}$. Since $\set{x_1,\dots, x_n}$ is clearly not a chain, we can assume that $x_1$ and $x_2$ are incomparable (in notation, $x_1\parallel x_2$) and  $x_1\vee x_2$ is a maximal element of $\set{x_i\vee x_j:  \set{i,j}\subseteq \set{1,\dots, n}}$. There is a $t\in\set{3,\dots,n}$ such that $x_t\not\leq x_1\vee x_2$ since otherwise we would have that $y=x_1\vee x_2$. 
We claim that $H:=\set{x_1\vee x_2, x_1\vee x_t, x_2\vee x_t}$ is a three-element antichain. Since $x_t\not\leq x_1\vee x_2$, 
we have that $x_i\vee x_t\not\leq x_1\vee x_2$ for $i\in\set{1,2}$. In particular, $x_i\vee x_t\neq x_1\vee x_2$. So if we had $x_1\vee x_2\leq x_i\vee x_t$, then  $x_1\vee x_2 < x_i\vee x_t$ would contradict the maximality of $x_1\vee x_2$. 
If we had that $x_1\vee x_t\nonparallel x_2\vee x_t$, say,  $x_1\vee x_t\leq  x_2\vee x_t$, then $x_1\vee x_2\leq (x_1\vee x_t)\vee(x_2\vee x_t)=x_2\vee x_t$ would lead to an already excluded case. So $H$ is a three-element antichain. We know from, say, Gr\"atzer~\cite[Lemma 73]{ggfoundbook} that $H$ generates a sublattice isomorphic to the eight-element boolean lattice. This contradicts the planarity of $L$ by, say, Kelly and Rival~\cite{kellyrival}.
\end{proof}

The next two lemmas will be needed later in the paper. 

\begin{lemma}[White \cite{whi}] 
Let $L$ be a semimodular lattice of finite length. If $\xi=(x_1,x_2)\in L^2$, then for all $x\in M(\xi)$, $x\leq x_1\vee x_2$.
\label{lem3}
\end{lemma}

Let $L$ be a lattice and $\xi=(x_1,\ldots,x_k)\in L^k$. Recall that $\{\xi\}$ denotes the set $\set{x_1,\ldots,x_k}$. Suppose $z\in L$ with $z\not\leq c_1(\xi)$. We note that for each $x_i\in\set\xi$ it is the case that $x_i\parallel z$ or $x_i<z$ . Let 
\begin{equation}\left.
\parbox{6.5cm}{
$\pxi=\{i: x_i\in \set\xi\ {\rm and}\ x_i\parallel z\}$ and\\ $\bxi=\{i: x_i\in \set\xi\ {\rm and}\ x_i<z\}$;}\,\,\,\right\}
\label{pbxbxNDpx}
\end{equation} 
the subscripts come from  ``parallel'' and ``below'', respectively.  Note that $|\pxi|+|\bxi|=k$.

\begin{lemma}
Let $L$ be a semimodular lattice of finite length. Let $\xi=(x_1,\ldots, x_k)\in L^k$ and $z\in L$ such that $z\not\leq c_1(\xi)$. If $|\pxi| \leq |\bxi|$, then $z\not\in M(\xi)$.
\label{AleqB}
\end{lemma}

\begin{proof}
If $|\pxi| = 0$, then $z > c_1(\xi)$. By Lemma 2.2 in \cite{cpw}, $z \not\in M(\xi)$. From now on we will assume that $|\pxi| \geq 1$ and so $z \parallel c_1(\xi)$. If $|\pxi|=|\bxi|=1$, then $z\not\in M(\xi)$ follows from Lemma \ref{lem3}. Assume that $|\bxi|\geq2$ and let $y:=\bigvee \{x_i\in \set\xi : x_i<z\}=\bigvee\set{x_i: i\in\bxi}$. Since $y\leq c_1(\xi)$, $y\leq z$, and $z\parallel c_1(\xi)$, it is the case that $y<z$. We observe that for each $x_i\in\set\xi$ with $x_i\parallel z$ (that is, for each $i\in\pxi$) the triangle inequality gives that
\begin{equation}
d(y,x_i)\leq d(y,z)+d(z,x_i),
\label{eqsTdTrngWlhWf}
\end{equation}
and for each $x_i\in\set\xi$ with $x_i<z$ (that is, for each $i\in\bxi$) \eqref{alignchainddd} implies that
\begin{equation}
d(y,x_i)=d(z,x_i)-d(y,z).
\label{equ1}
\end{equation}
We may assume without loss of generality that 
$1\in\pxi$ and so $x_1\parallel z$. Note that $y\vee x_1\leq z\vee x_1$. Since $y\vee x_1\leq c_1(\xi)$ and $z\vee x_1\not\leq c_1(\xi)$ it follows that $y\vee x_1<z\vee x_1$. Thus 
\begin{equation}
d(y,y\vee x_1)<d(y,z\vee x_1)\ {\rm and}\ d(y\vee x_1,x_1)<d(z\vee x_1,x_1).
\label{eqstDvSkzgh}
\end{equation}
We may assume that $2\in\bxi$ and so $x_2<z$. Using \eqref{eqPndskszD}, \eqref{eqstDvSkzgh}, and  the triangle inequality at $\leq'$, we get    
\allowdisplaybreaks{
\begin{align}
d(y,x_1)+d(y,x_2)  \eqeqref{eqPndskszD} {} &d(y,y\vee x_1)+d(y\vee x_1,x_1)+d(y,x_2)\cr
                  \lesseqref{eqstDvSkzgh} {} & d(y,z\vee x_1) + d(z\vee x_1,x_1) +  d(y,x_2)  \cr
                   \leq' \kern 2pt {} & d(y,z)+ d(z,z\vee x_1)+ d(z\vee x_1,x_1)+ d(y,x_2) \cr
                   \eqeqref{eqPndskszD} {} &  d(z,x_1) +  d(z,y)+ d(y,x_2) \cr
                 \eqeqref{alignchainddd} {} & d(z,x_1)+d(z,x_2),\qquad\text{ whereby}\cr
d(y,x_1)+d(y,x_2)<\phantom{'}  {} & d(z,x_1)+d(z,x_2).
\label{alignlLn12}
\end{align}
}

Finally, let $\pxi'=\pxi\setminus\{1\}$ and let $\bxi'=\bxi\setminus\{2\}$. Using the  inequality $|\pxi'| \leq |\bxi'|$ at $\leq'$, we get  the following calculation.
\allowdisplaybreaks{
\begin{align*}
r(y,\xi)  =  {}& \sum_{i\in\pxi} d(y,x_i)+\sum_{i\in\bxi} d(y,x_i)\\
           =  {}& \sum_{i\in\pxi'} d(y,x_i)+ d(y,x_1) +\sum_{i\in\bxi'} d(y,x_i)+d(y,x_2)\\
             \overset{(\ref{eqsTdTrngWlhWf},\ref{equ1})}\leq   {}&\sum_{i\in\pxi'} d(z,x_i)+|\pxi'|\cdot d(y,z)+d(y,x_1){}+{}\\
            {}& \sum_{i\in\bxi'} d(z,x_i)-|\bxi'|\cdot d(z,y)+d(y,x_2)  \\
          \leq'{}& 	 \sum_{i\in\pxi'} d(z,x_i)+d(y,x_1) + \sum_{i\in\bxi'} d(z,x_i)+d(y,x_2) \\
           \lesseqref{alignlLn12}  {}&  \sum_{i\in\pxi'} d(z,x_i)+ d(z,x_1) +\sum_{i\in\bxi'} d(z,x_i)+d(z,x_2)  =   r(z,\xi).
\end{align*}
}
Hence $r(y,\xi)<r(z,\xi)$, and so $z\not\in M(\xi)$, as required.
\end{proof}

Note that in  the proof of Lemma~\ref{lemmaProdBr}, where  $\xi$ is given in  \eqref{eqmtxSxTdnQ} modulo notational changes  and $z=(0,\ldots,0,n-1,1)$, we have $|\pxi|=2>1=|\bxi|$. Therefore the restriction $|\pxi|\leq |\bxi|$ given in Lemma \ref{AleqB} cannot be dropped.

\goodbreak

\section{Main result}\label{sectmainresult}

In harmony with the general convention that the empty join is the least element, note that 
the breadth of the singleton lattice is $0$.

\begin{theorem}\label{thmmain}\ 
\begin{enumeratei}
\item\label{thmmaina} Let $L$ be a semimodular lattice of finite length. If $L$ is of breadth at most $2$, then  $L$ satisfies the $c_1$-median property.
\item\label{thmmainb} For each integer $k\geq 3$, there exists a finite semimodular lattice of breadth $k$ that  fails to satisfy the $c_1$-median property. 
\item\label{thmmainc} Let $t$ be a positive integer. For $i=1,\dots,t$, let $L_i$ be a lattice of finite length satisfying the $c_1$-median property. Then the direct product $L:=L_1\times\dots\times L_t$ is a  lattice of finite length  and it also satisfies the $c_1$-median property. If all the $L_i$ are of finite breadth, then $\br(L)=\br(L_1)+\dots+\br(L_t)$. 
Furthermore, if all the $L_i$ are semimodular, then so is $L$.
\end{enumeratei}
\end{theorem}

\begin{proof} In order to prove part \eqref{thmmaina}, 
let $L$ be a semimodular lattice of finite length with breadth 2. Let $\xi=(x_1,\ldots, x_k)\in L^k$ and $z\in L$ with $z\not\leq c_1(\xi)$; we need to show that $z\notin M(\xi)$. 
If $k=2$, then $z\not\in M(\xi)$ follows from Lemma \ref{lem3}. From now on we will assume that $k\geq3$.
With the notation of \eqref{pbxbxNDpx}, 
$|\pxi|\leq|\bxi|$ implies $z\not\in M(\xi)$ by Lemma \ref{AleqB}.
Now suppose that $|\pxi|>|\bxi|$. Consider the set $T=\{z\vee x_i : i\in\pxi \}$. Let $z\vee x_i,z\vee x_j\in T$. Breadth 2 implies that $(z\vee x_i)\vee (z\vee x_j)=z\vee x_i\vee x_j\in \{x_i\vee x_j,z\vee x_i,z\vee x_j\}$. Note that $z\vee x_i\vee x_j=x_i\vee x_j$ would imply that $z<x_i\vee x_j\leq c_1(\xi)$, a contradiction. So $(z\vee x_i)\vee (z\vee x_j)\in \{z\vee x_i,z\vee x_j\}$. 
Thus  $T$ is a chain; let $z\vee x_j$ be its least element.

We claim that for each $x_i\in\set\xi$ with  $x_i\parallel z$ (that is, for each $i\in\pxi$), 
\begin{equation}
d(z\vee x_j,x_i)\leq d(z,x_i)-d(z,z\vee x_j).
\label{eqbzHfdskWg}
\end{equation}
To see this consider that for each $i\in\pxi$ we have that
\begin{align*}
d(z,x_i) & \eqeqref{eqPndskszD}   d(z,z\vee x_i)+d(z\vee x_i,x_i)\\
           & \eqeqref{alignchainddd}  d(z,z\vee x_j)+d(z\vee x_j,z\vee x_i)+d(z\vee x_i,x_i).
\end{align*}
Hence $d(z,x_i)-d(z,z\vee x_j)  =   d(z\vee x_j,z\vee x_i)+d(z\vee x_i,x_i)$, which implies \eqref{eqbzHfdskWg} by the triangle inequality.
Further, for each $x_i\in\set\xi$ with $x_i< z$ (that is, for $i\in\bxi$),
\begin{equation}
d(z\vee x_j,x_i) \eqeqref{alignchainddd}  d(z,x_i)+d(z,z\vee x_j)
\label{eqwBbTktlBgrtZ}
\end{equation}
since $x_i<z < z\vee x_j$. Armed with \eqref{eqbzHfdskWg} and \eqref{eqwBbTktlBgrtZ}, we have that
\begin{align*}
r(z\vee x_j,\xi)  =  {}& \sum_{i\in\pxi} d(z\vee x_j,x_i)+ \sum_{i\in\bxi} d(z\vee x_j,x_i)\\
            \leq    {}& \sum_{i\in\pxi} d(z,x_i)-|\pxi|\cdot d(z,z\vee x_j){} + {}\\
              {}& \sum_{i\in\bxi} d(z,x_i)+|\bxi|\cdot d(z,z\vee x_j)\\
            =   {}&   r(z,\xi)-d(z,z\vee x_j)\cdot (|\pxi|-|\bxi|)\\
            <    {}&  r(z,\xi) \hspace{0.3in} \mbox{(since\ $d(z,z\vee x_j)>0$\ and\ $|\pxi|>|\bxi|$).}
\end{align*}
Hence $r(z\vee x_j,\xi)<r(z,\xi)$, and so $z\not\in M(\xi)$.
This proves part \eqref{thmmaina}.

Part \eqref{thmmainb} of the theorem follows from Lemma~\ref{lemmaProdBr} or from Remark~\ref{remarkGlS}.

Next, to prove part \eqref{thmmainc}, assume that $L:=L_1\times\dots\times L_t$ such that  $L_i$ is a lattice of finite length satisfying the $c_1$-median property for $i=1,\dots,t$. Clearly, we can assume that $t=2$ since then the case $t>2$ follows by a trivial induction. So, $L=L_1\times L_2$. We can assume that none of $L_1$ and $L_2$ is a singleton. We claim that  for any $x=(x_1,x_2)$ and $y=(y_1,y_2)$ in $L$, 
\begin{equation}
d(x,y)=d(x_1,y_1)+d(x_2,y_2).
\label{eqdStcmPnWs}
\end{equation}
To prove this, let $n:=d(x_1,y_1)$ and $m:=d(x_2,y_2)$. 
The \emph{neighboring relation} ``$\prec$'' $\cup$ ``$\succ$'' will be denoted by $\prec^\ast$.
By the definition of our distance function 
$d$, there are sequences
$x_1=a_0$, $a_1$, \dots, $a_n=y_1$ in $L_1$ and 
$x_2=b_0$, $b_1$, \dots, $b_m=y_2$ in $L_2$ such that
$a_i \prec^\ast_{L_1} a_{i+1}$  for all $i<n$ and $b_j\prec^\ast_{L_2} b_{j+1}$ for all $j<m$. 
Since the pair of any two consecutive members  of the sequence $x=(x_1,x_2)=(a_0,b_0)$, $(a_1,b_0)$, \dots $(a_n,b_0)$, $(a_n,b_1)$, \dots, $(a_n,b_m)=(y_1,y_2)=y$ belongs to  $\prec^\ast$, we obtain that  $d(x,y)\leq m+n=d(x_1,y_1)+d(x_2,y_2)$. 
Conversely, let $x=(x_1,x_2)=(u_0,v_0)$, $(u_1,v_1)$, \dots, $(u_s,v_s)=(y_1,y_2)=y$ be a sequence in $L$ such that
the pairs of its consecutive members 
belong to $\prec^\ast$. Let 
\begin{align*}
A&:=\set{i: 0\leq i<s,\,\, u_i \prec^\ast_{L_1} u_{i+1},\,\,  v_i = v_{i+1} } \text{ and}\cr
B&:=\set{i: 0\leq i<s,\,\, v_i \prec^\ast_{L_2} v_{i+1},\,\,  u_i = u_{i+1} }.
\end{align*}
It follows from \eqref{pbxPrCrlPr} that $\set{1,2,\dots, s}$ is the disjoint union of $A$ and $B$. In particular, $|A|+|B|=s$. Observe that  $\set{u_i: i\in A}$ is a sequence of $\prec^\ast_{L_1}$-neighboring elements 
from $x_1$ to $y_1$; for example, if $s=7$ and  $A=\set{2,4,5}$,  then this sequence is
$x_1=u_0=u_1=u_2\prec^\ast u_3=u_4\prec^\ast u_5\prec^\ast u_6=u_7=y_1$. Hence,  $n=d(x_1,y_1)\leq |A|$.
Similarly, $m=d(x_2,y_2)\leq |B|$. 
Thus  
$s=|A|+|B|\geq d(x_1,y_1) + d(x_2,y_2)$, and we conclude that $d(x,y)\geq d(x_1,y_1) + d(x_2,y_2)$, proving \eqref{eqdStcmPnWs}.

Next, for an arbitrary profile $\xi=(x\ssc1,\dots,x\ssc k)\in L^k$ and  $i\in\set{1,2}$, we let 
$\xi_i:=(x\ssc1_i, \dots, x\ssc k_i)\in L_i^k$. 
For every $y\in L$, \eqref{eqdStcmPnWs} gives that
\begin{equation}
r(y,\xi)=r(y_1,\xi_1) + r(y_2,\xi_2).
\label{eqzbtrrnHlsQx}
\end{equation}
Now assume that $y\in M(\xi)$, that is, $r(y,\xi)$ is minimal for \emph{this} $\xi$. Let $i\in\set{1,2}$.
If $r(y_1,\xi_1)$ was not minimal for $\xi_1$, then we could pick an element $y_1'\in L_1$ with $r(y_1',\xi_1)<r(y_1,\xi_1)$, we could take $\widehat y:=(y_1',y_2)$ in $L$, and we would have $r(\widehat y,\xi)<r(y,\xi)$ by \eqref{eqzbtrrnHlsQx}, contradicting the minimality of $r(y,\xi)$. Hence, $r(y_1,\xi_1)$ is minimal and $y_1\in M(\xi_1)$. 
Since the indices  1 and 2 play a symmetric role, we obtain in the same way that  $y_2\in M(\xi_2)$. 
Since $L_i$ satisfies the $c_1$-median property for $i\in\set{1,2}$, we obtain that $y_i\leq c_1(\xi_i)=x\ssc 1_i\vee \dots\vee x\ssc k_i$. Consequently,
$y\leq x\ssc 1\vee \dots \vee \ssc k$, which proves that $L$ satisfies the $c_1$-median property. 
The assertion on $\br(L)$ is Lemma~\ref{lemmaProDct}. 
Finally, \eqref{eqtxtWsrptRdzd} completes the proof of Theorem~\ref{thmmain}.
\end{proof}


\end{document}